\def\draft{y}
\documentclass[12pt]{amsart}
\usepackage[headings]{fullpage}
\usepackage{amssymb,epic,eepic,epsfig}
\usepackage{graphicx}
\usepackage{texdraw}
\usepackage{url}
\usepackage[bookmarks=true,%
    colorlinks=true,%
    linkcolor=blue,%
    citecolor=blue,%
    filecolor=blue,%
    menucolor=blue,%
    urlcolor=blue,%
    breaklinks=true]{hyperref}

\newtheorem{theorem}{Theorem}[section]
\newtheorem{proposition}{Proposition}[section]
\theoremstyle{definition}
\newtheorem{lemma}[proposition]{Lemma}
\newtheorem{definition}[proposition]{Definition}
\newtheorem{remark}[proposition]{Remark}
\newtheorem{corollary}[proposition]{Corollary}

\def\printname#1{
        \if\draft y
                \smash{\makebox[0pt]{\hspace{-0.5in}
                        \raisebox{8pt}{\tt\tiny #1}}}
        \fi
}

\newcommand{\psdraw}[2]
         {\begin{array}{c} \hspace{-1.3mm}
        \raisebox{-4pt}{\epsfig{figure=draws/#1.eps,width=#2}}
        \hspace{-1.9mm}\end{array}}

\newlength{\standardunitlength}
\catcode`\@=11
\long\def\@makecaption#1#2{%
     \vskip 10pt

\setbox\@tempboxa\hbox{
       \small\sf{\bfcaptionfont #1. }\ignorespaces #2}%
     \ifdim \wd\@tempboxa >\captionwidth {%
         \rightskip=\@captionmargin\leftskip=\@captionmargin
         \unhbox\@tempboxa\par}%
       \else
         \hbox to\hsize{\hfil\box\@tempboxa\hfil}%
     \fi}
\font\bfcaptionfont=cmssbx10 scaled \magstephalf
\newdimen\@captionmargin\@captionmargin=2\parindent
\newdimen\captionwidth\captionwidth=\hsize
\catcode`\@=12

\newcommand{\tr}{\operatorname{tr}}

\def\lbl#1{\label{#1}\printname{#1}}

\def\BN{\mathbb N}
\def\BZ{\mathbb Z}

\def\BQ{\mathbb Q}

\def\BR{\mathbb R}
\def\BC{\mathbb C}

\def\a{\alpha}

\def\l{\lambda}

\def\w{\omega}

\def\e{\epsilon}

\def\d{\delta}
\def\b{\beta}

\def\longto{\longrightarrow}
\def\w{\omega}
\def\SL{\mathrm{SL}}
\def\PSL{\mathrm{PSL}}

\def\w{\omega}

\def\calA{\mathcal{A}}

\def\d{\delta}

\def\fg{\mathfrak{g}}
\def\coeff{\mathrm{coeff}}

\begin{document}


\title[The Newton polygon of a recurrence sequence
and its role in TQFT]{The Newton polygon of a recurrence sequence 
 of polynomials and its role in TQFT}
\author{Stavros Garoufalidis}
\address{School of Mathematics \\
         Georgia Institute of Technology \\
         Atlanta, GA 30332-0160, USA \newline 
         {\tt \url{http://www.math.gatech.edu/~stavros }}}
\email{stavros@math.gatech.edu}
\thanks{The author was supported in part by NSF. \\
\newline
1991 {\em Mathematics Classification.} Primary 57N10. Secondary 57M25.
}

\date{October 24, 2012}


\begin{abstract}
The paper contains a combinatorial theorem (the sequence of Newton polygons
of a reccurent sequence of polynomials is quasi-linear) and two applications 
of it in classical and quantum topology, namely in the behavior of the 
$A$-polynomial and a fixed quantum invariant (such as the Jones polynomial)
under filling. Our combinatorial theorem, which complements results of 
Calegari-Walker \cite{CW} and the author \cite{Ga4}, 
occupies the bulk of the paper and 
its proof requires the Lech-Mahler-Skolem theorem of $p$-adic analytic number
theory combined with basic principles in polyhedral and tropical geometry.
\end{abstract}

\maketitle

\tableofcontents

\section{Introduction}
\lbl{sec.intro}

\subsection{Filtered chain complexes versus holonomy}
\lbl{sub.summary}

Filtered chain complexes and their associated spectral sequences and
exact triangles are standard tools of Homological Algebra that have found
numerous applications in the deep categorification theories of Khovanov, 
Kronheimer-Mrowka, Ozsv\'ath-Sz\'abo and many others; 
see \cite{Kh,KM,OS}. For instance the genus of
a knot in 3-space can be effectively computed by the Knot Homology of 
\cite{OS}. 

On the other hand, one has the TQFT invariants 
of knotted 3-dimensional objects, and a good example to keep in mind
is the famous Jones polynomial of a knot; see \cite{Jo}. 
There are several known and conjectured connections between the (colored) 
Jones polynomial of a knot and the geometry and topology of the 
knot complement. In particular, the colored
Jones polynomial determines the Alexander polynomial (\cite{B-NG,GL2}), and is 
conjectured to determine 
\begin{itemize}
\item[(a)]
the Volume of the knot (\cite{Ka,MM}), 
\item[(b)] 
the $A$-polynomial of the knot via the AJ-Conjecture (\cite{Ga1}), 
\item[(c)]
at least two boundary slopes of incompressible surfaces of the knot complement
via the Slope Conjecture (\cite{Ga4}), and
\item[(d)]
the invariant trace field of a hyperbolic knot, via the
subleading asymptotics to the Volume Conjecture (\cite{Ga2,DGLZ,GZ})
\end{itemize}
Knot Homology and TQFT have their own strengths. A major advantage of Knot 
Homology is its functorial nature, which numerical TQFT invariants (such
as the Jones polynomial) lack. What concept plays the role of functoriality
in TQFT? We argue that the notions of {\em holonomy} and {\em $q$-holonomy}
play this role in TQFT. We illustrate this principle with two independent 
results.
\begin{itemize}
\item
We study the behavior of the $A$-polynomial and quantum 
invariants (such as the Jones or Alexander polynomials) under 1-parameter 
fillings of a 2-cusped manifold, see Theorems \ref{thm.1} and \ref{thm.2}
below. In the $A$-polynomial case, it divides a holonomic sequence of 
polynomials in 2-variables, and in the quantum invariant case, it is a 
holonomic sequence of polynomials in one variable. 
\item
We prove that the Newton polygon $N_n$ of a holonomic sequence of polynomials 
is quasi-linear; see Theorem \ref{thm.3}. This complements results 
of \cite{CW} and \cite{Ga4}.
\end{itemize}
Holonomy (and $q$-holonomy) was studied
extensively by Zeilberger; see \cite{Z}. 
$q$-holonomic sequences of polynomials in one variable appeared in Quantum 
Topology in \cite{GL1}, where it was shown that the colored Jones function 
of a knot and an arbitrary simple Lie algebra is $q$-holonomic; see 
\cite[Thm.1]{GL1}.

Holonomic sequences of multivariable polynomials are easier to analyze, 
and they
appear naturally when one studies geometrically similar families of knots,
such as those that arise from filling of all but one components of a link.
This is not difficult to prove, but it is not widely known. 
Let us recall what is a holonomic sequence following Zeilberger; see \cite{Z}. 
Let $K=\BQ(x_1,\dots,x_r)$ denote the field of rational functions in $r$ 
variables $x_1,\dots,x_r$.

\begin{definition}
\lbl{def.holo}
We say that a sequence of rational functions $R_n \in K$ 
(defined for all integers $n$) is {\em recurrent} (or,
{\em constant coefficient holonomic}) if it satisfies a linear 
recursion with constant coefficients. In other words, there exists a natural
number $d$ and $c_k \in K$ for $k=0,\dots,d$ with $c_d \neq 0$
such that for all integers $n$ we have:
\begin{equation}
\lbl{eq.recRR}
\sum_{k=0}^d c_k R_{n+k}=0
\end{equation}
\end{definition}
Depending on the circumstances, one can restrict attention to sequences
indexed by the natural (rather than the integer) numbers. Please note that
the coefficients of the recursion \eqref{eq.recRR} are constant, independent
of $n$. If we allow them to polynomially depend on $n$, the corresponding 
sequence is by definition {\em holonomic}. Our two applications discussed in
Theorems \ref{thm.1} and \ref{thm.2} below concern reccurent sequences.

\subsection{The behavior of the $A$-polynomial under filling}
\lbl{sub.filling}

The $A$-polynomial $A_M$ of an oriented hyperbolic 3-manifold $M$ with 
one cusp was introduced in \cite{CCGLS}. We will assume that $A_M$ parametrizes
a {\em geometric component} of the $\SL(2,\BC)$ character variety.
$A_M$ is a 2-variable polynomial 
which describes the dependence of the eigenvalues of a meridian and longitude
under any representation of $\pi_1(M)$ into $\SL(2,\BC)$. 
The $A$-polynomial 
plays a key role in two problems:
\begin{itemize}
\item
the deformation of the hyperbolic structure of $M$,
\item
the problem of exceptional (i.e., non-hyperbolic) fillings of $M$.
\end{itemize}
Knowledge about the $A$-polynomial (and often, of its Newton polygon)
is translated directly into information of the above problems, and vice-versa.
This key property was explained and exploited by Culler-Shalen and 
Gordon-Luecke in \cite{CGLS,CCGLS,Go}. Technically, the $\SL(2,\BC)$ character
variety of $M$ has several components, and the unique discrete faithful
$\PSL(2,\BC)$ representation of the oriented manifold $M$ always lifts to 
as many $\SL(2,\BC)$ representations as the order of the finite group
$H_1(M,\BZ/2)$; see \cite{Cu}. When $M$ is a complement of a hyperbolic knot 
in an oriented integer homology sphere, the $A$-polynomial of $M$ is defined 
to be the $\SL(2,\BC)$ lift of the geometric $\PSL(2,\BC)$
component of $M$. 

Our first goal is to describe the behavior of the $A$-polynomial
under filling one of the cusps of a 2-cusped hyperbolic 3-manifold.
To state our result, consider an oriented hyperbolic 3-manifold $M$ which
is the complement of a hyperbolic link in a homology 3-sphere.
Let $(\mu_1,\l_1)$ and $(\mu_2,\l_2)$ denote pairs of meridian-longitude curves
along the two cusps $C_1$ and $C_2$ of $M$, and let $M_n$ denote the result of 
$-1/n$ filling on $C_2$. Let $A_n(m_1,l_1)$ denote the $A$-polynomial of $M_n$
with the meridian-longitude pair $(\mu_1,\l_1)$ inherited from $M$.

\begin{theorem}
\lbl{thm.1}
For every 3-manifold $M$ as above, there exists a recurrent sequence 
$R_n(m_1,l_1) \in \BQ(m_1,l_1)$ such that 
for all but finitely many integers $n$, $A_n(m_1,l_1)$ divides the 
numerator of $R_n(m_1,l_1)$.
In addition, a recursion for $R_n$ can be computed explicitly via elimination,
from an ideal triangulation of $M$.
\end{theorem}

\subsection{The behavior of Quantum Invariants under filling}
\lbl{sub.appqt}

In Section \ref{sub.filling} we showed how holonomic sequences arise
in 1-parameter families of character varieties. In this section, we
will show how they arise in Quantum Topology.
Consider two endomorphisms $A,B$ of a finite dimensional vector space
$V$ over the field $\BQ(q)$. 
Let $\tr(D)$
denote the {\em trace} of an endomorphism $D$. The next lemma is
an elementary application of the {\em Cayley-Hamilton} theorem.

\begin{lemma}
\lbl{lem.AB}
With the above assumptions, the sequence $\tr(AB^n) \in \BQ(q)$ is
recurrent.
\end{lemma}
We need to recall the relevant Quantum Invariants of links from 
\cite{Jo,Jn,Tu}. Fix a simple Lie algebra $\fg$, a representation $V$ of
$\fg$, a knot $K$, and consider the {\em Quantum Group invariant} 
$Z^{\fg}_{V,K}(q)  \in \BZ[q^{\pm 1}]$. For instance,
\begin{itemize}
\item
When $\fg=\mathfrak{sl}_2$, and $V=\BC^2$ is the defining representation,
$Z^{\fg}_{V,K}(q)$ is the Jones polynomial of $K$. 
\item
When $\fg=\mathfrak{gl}(1|1)$ and $V=\BC^2$, $Z^{\fg}_{V,K}(q)$ is the Alexander
polynomial of $K$.
\end{itemize}
The quantum group invariant $Z^{\fg}_{V,K}(q)$ can be computed as the trace
of an operator associated to a braid presentation of $K$.
  
Let $L$ denote a 2-component link in $S^3$ with one unknotted component
$C_2$, and let $K_n$ denote the knot obtained by $-1/n$ filling on $C_2$.
Since $S^3\setminus C_2$ is a solid torus $S^1 \times D^2$ and $L$ is a knot in 
$S^1 \times D^2$ it follows that $L$ is the closure of an $(r,r)$ tangle 
$\a$. If $\b$ is a braid representative of a full positive twist
in the braid group of $r$ stands, then it follows that $K_n$ is obtained
by the closure of the tangle $\a\b^n$. 
If $A$ and $B$ denote the endomorphism of 
$V$ corresponding to $\a$ and $\b$ 
and $Q(n)$ denotes half of the negative
of the writhe of $\a\b^n$ multiplied by the quadratic Casimir value of $\fg$
in $V$, then we have:
$$
Z_{V,K_n}(q)=q^{-Q(n)}\tr(A B^n)
$$
Notice that $Q(n)$ is a quadratic function of $n$, and its presence
in the above formula is required if we insist that $K_n$ is a zero-framed
knot. The next theorem follows from the above discussion, Lemma \ref{lem.AB}
and Theorem \ref{thm.3}.

\begin{theorem}
\lbl{thm.2}
Fix a simple Lie algebra $\fg$ and a representation $V$ of $\fg$.
With the above assumptions, the sequence $q^{Q(n)}Z^{\fg}_{V,K_n}(q) \in \BQ(q)$ 
is recurrent. 
\end{theorem}

\subsection{The sequence of Newton polygons of a 
recurrent sequence of polynomials}
\lbl{sub.newton}

Often one is interested in the Newton polygon of the $A$-polynomial;
for instance the slopes of its sides are boundary slopes of incompressible
surfaces as follows by Culler-Shalen theory; see \cite{CS1,CS2}.
This is one motivation of our next result. Recall that a {\em quasi-polynomial}
$p(n)=\sum_{j=0}^d a_j(n) n^j$ is a polynomial in one variable $n$ 
with coefficients $a_j(n)$ periodic functions of $n$. We will call $p(n)$
quasi-linear (resp. quasi-quadratic) if $d \leq 1$ (resp., $d \leq 2$).
Quasi-polynomials appear in Enumerative Combinatorics (see \cite{St}), 
and also in the lattice point counting problems via Ehrhart's theorem 
\cite{Eh}. Quasi-polynomials have appeared recently in the work of 
Calegari-Walker (see \cite{CW}) and also in Quantum Topology in relation
to the Slope Conjecture; see \cite{Ga3,Ga4}.

\begin{definition}
We say that a sequence $N_n$ of polygons is linear 
(resp. quasi-linear) if the coordinates of the vertices of $N_n$ are
polynomials (resp. quasi-polynomials) of degree at most one. Likewise,
we say that a sequence  $N_n$ of polygons is quadratic 
(resp. quasi-quadratic) if the coordinates of the vertices of $N_n$ are
polynomials (resp. quasi-polynomials) of degree at most two.
\end{definition}
The next theorem is of independent interest and explains the first
part of the title of the paper. Its proof follows from first principles of 
Polyhedral and Tropical Geometry, and the {\em Lech-Mahler-Skolem theorem} of
Analytic Number Theory.

\begin{theorem}
\lbl{thm.3}
Let $N_n$ be sequence of Newton polygons of a holonomic sequence $R_n \in
\BQ[x_1^{\pm 1},\dots,x_r^{\pm 1}]\setminus\{0\}$.
Then, for all but finite many integers $n$,
$N_n$ is quasi-linear. 
\end{theorem}
The next corollary of Theorem \ref{thm.3} follows from some recent
results of Chen-Li-Sam which generalize the Ehrhart theory; see \cite{CLS}.

\begin{corollary}
\lbl{cor.2}
Under the hypothesis of Theorem \ref{thm.3}, the volume and the number
of lattice points of $R_n$ is a quasi-polynomial function of $n$.
\end{corollary}

\begin{remark}
\lbl{rem.Rnonzero}
The hypothesis that $R_n \neq 0$ for all $n$ in Theorem \ref{thm.3}
is not as strict as it seems (and is trivially satisfied in its application
to Theorems \ref{thm.1} and \ref{thm.2}). Indeed, the Skolem-Mahler-Lech 
theorem implies that the zero set of any holonomic sequence in a field of 
characteristic zero vanishes on a finite union of full arithmetic progressions
minus a finite set. See Theorem \ref{thm.MLS} below.
\end{remark}

\begin{remark}
\lbl{rem.finite}
The reader may have noticed that Theorems \ref{thm.1} and 
Theorem \ref{thm.3} are valid for all but finitely many integers $n$.
There are two independent sources for this exception. In Theorem \ref{thm.1}, 
the exception comes from the Hyperbolic Dehn filling theorem of Thurston 
(see \cite{Th} and also \cite{NZ}), which implies that all but finitely 
many fillings in one cusp of a hyperbolic 3-manifold gives a hyperbolic
manifold. The finite set of exceptional fillings are the focus of the 
problem of {\em exceptional Dehn surgery} of \cite{CGLS}. 
In Theorem \ref{thm.3}, the exceptions come from the 
Skolem-Mahler-Lech theorem, which states that the zeros of a sequence of
rational numbers that satisfies a constant coefficient recursion relation
are full sets of arithmetic progressions, up to a finite exceptional
set. Is there a connection between exceptional Dehn surgery and the
LMS theorem?
\end{remark}

\subsection{Application in character varieties}
\lbl{sub.appcv}

Theorems \ref{thm.1} and \ref{thm.3} are general, but in favorable
circumstances more is true. Namely, consider a family of knot complements 
$K_n$, obtained by $-1/n$ filling on a cusp of 2-component hyperbolic 
link $L$ in $S^3$, with linking number $f$. 
Let $A_n(m,l)$ denote the $A$-polynomial of $K_n$ with respect to
the canonical meridian and longitude $(\mu,\l)$ of $K_n$.

\begin{definition}
\lbl{def.favor}
We say that two component hyperbolic $L$ link in $S^3$ with linking number $f$ 
is {\em favorable} if $A_n(m,l m^{-f^2 n}) \in \BQ[m^{\pm 1},l^{\pm 1}]$ is
holonomic, for all but finitely many values of $n$.
\end{definition}
The shift $l \mapsto l m^{-f^2 n}$ in the above coefficients
is due to the fact that the canonical meridian-longitude pair of $K_n$
differs from the corresponding pair of the unfilled component of $L$
due to the nonzero linking number. 

\begin{remark}
\lbl{rem.quadshift}
If $N_n$ is a sequence of Newton polygons in $\BR^2$ (with coordinates
$(m,l)$ which is quasi-linear, then the sequence of polygons obtained by
applying the transformation $(m,l)\mapsto (m,l m^{-f^2 n})$ to $N_n$ is
quasi-quadratic. Indeed, there is a finite set $I$ and quasi-linear functions
$s_i(n), t_i(n)$ for $i \in I$ such that $N_n$ is the convex hull of 
$(s_i(n),t_i(n))$ for $i \in I$. The monomial transformation 
$(m,l)\mapsto (m,l m^{-f^2 n})$ is the linear map $(a,b) \mapsto (a-f^2b n,b)$
of $\BR^2$ which sends $(s_i(n),t_i(n))$ to $(s_i(n)-f^2 t_i(n)n,t_i(n))$.
Since $s_i$ and $t_i$ are quasi-linear, it follows that $s_i(n)-f^2 t_i(n)n$
and $t_i(n)$ are quasi-quadratic. 
\end{remark}

Our next corollary gives a natural source of sequences of quasi-quadratic 
polytopes that come from classical topology, i.e., character varieties of knot 
complements. As discussed in \cite{CW}, this class of polytopes is a natural
generalization of the sequence $nP$ of scalings of a fixed rational
polytope $P$.

\begin{corollary}
\lbl{cor.3}
If a two component link $L$ is favorable, then for all but finitely
many $n$, the Newton polygon of $A_n(m,l)$ is quasi-quadratic.
\end{corollary}
Favorable links are more common than one might think.
Let us show three favorable links
$$ 
\psdraw{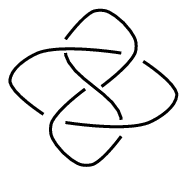}{1.5in} \qquad
\psdraw{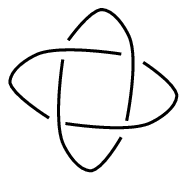}{1.5in} \qquad
\psdraw{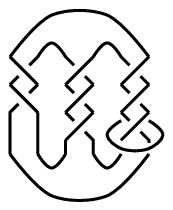}{1.3in}
$$
the {\em Whitehead link} (in the left), the {\em twisted Whitehead link} 
(in the middle) and the pretzel link (in the right). The first two links
were shown to be favorable by Hoste-Shanahan (see \cite{HS}) and surgeries
gives rise to two families of twist knots. The last link was shown to
be favorable in \cite{GM} and surgeries gives rise to the family of 
$(-2,3,3+2n)$ pretzel knots.

\subsection{Acknowledgment}
The author wishes to thank N. Dunfield, T.T.Q. Le, T. Mattman and J. Yu 
for useful conversations.

\section{Proof of Theorem \ref{thm.1}}
\lbl{sec.thm1}

Fix an oriented hyperbolic 3-manifold $M$ with two cusps 
$C_1$ and $C_2$ and choice of meridian-longitude $(\mu_i,\l_i)$ on each cusp 
for $i=1,2$. Let $K_n$ denote the result of $-1/n$ filling on $C_2$.
We consider two cases: $M$ has strongly geometrically isolated cusps,
or not. For a definition of strong geometric isolation, see \cite{NR} 
and also \cite{Ca,CW}. When $M$ is strongly geometrically isolated, Dehn
filling on one cusp does not change the shape of the other. This implies
that $A_n$ is constant (and therefore, holonomic), for all but finitely 
many integers $n$. Thus, we may assume that $M$ does not have strongly
geometrically isolated cusps.

Let $X$ denote the geometric component of the $\SL(2,\BC)$ character variety
of $M$. The {\em hyperbolic Dehn filling} theorem of Thurston implies that $X$
is a complex affine surface; see \cite{Th} and also \cite{NZ}. 
So, the field $F$ of rational functions on $X$ 
has transendence degree 2. Now $X$ has four known nonconstant
rational functions: the eigenvalues of the meridians $m_1,m_2$ and the 
longitudes $l_1,l_2$ around each cusp. So, each triple $\{m_1,l_1,m_2\}$ and 
$\{m_1,l_1,l_2\}$ of elements of $F$ is polynomially dependent i.e., satisfies
a polynomial equation
\begin{equation}
\lbl{eq.PQ}
P(m_1,l_1,m_2)=0 \qquad Q(m_1,l_1,l_2)=0 
\end{equation}
where $P(m_1,l_1,m_2) \in \BQ(m_1,l_1)[m_2]$ and 
$Q(m_1,l_1,l_2) \in \BQ(m_1,l_1)[l_2]$ are polynomials of strictly positive
(by hypothesis) degrees $d_P$ and $d_Q$ with respect to 
$m_2$ and $l_2$. The geometric component $X_n$ of the character variety of $K$
is the intersection of $X$ with the Dehn-filling equation $m_2 l_2^{-n}=1$
\cite{Th}. 
So, on $X_n$ we have $P(m_1,l_1,l_2^n)=0$. 
Let $p(m_1,l_1)$ and $q(m_1,l_1)$ denote the 
coefficient of $m_2^{d_P}$ and $l_2^{d_Q}$ in $P(m_1,l_1,m_2)$
and $Q(m_1,l_1,l_2)$ respectively. Let $R_n(m_1,l_1) \in \BQ(m_1,l_1)$ 
denote the {\em resultant} of $P(m_1,l_1,l_2^n)$ and $Q(m_1,l_1,l_2)$ 
(both are elements of $\BQ(m_1,l_1)[l_2]$)
with respect to $l_2$; see \cite[Sec.IV.8]{La}. It follows that 
$$
R_n(m_1,l_1)= p(m_1,l_1)^{d_Q} \prod_{l_2: Q(m_1,l_1,l_2)=0} P(m_1,l_1,l_2^n)
\in \BQ(m_1,l_1)
$$
Since $R_n(m_1,l_1)$ is a $\BQ(m_1,l_1)$ linear combination of 
$P(m_1,l_1,l_2^n)$ and $Q(m_1,l_1,l_2)$ (see \cite[Sec.IV.8]{La}) and since
$P(m_1,l_1,l_2^n)$ and $Q(m_1,l_1,l_2)$ vanish on the irreducible curve
$X_n$, it follows that $A_n(m_1,l_1)$ divides the numerator of $R_n(m_1,l_1)$.
Moreover, by the above equation, $R_n(m_1,l_1)$ 
is a $\BQ(m_1,l_1)$-linear combination of the $n$-th powers of a 
finite set of elements $l_2$ algebraic over $\BQ(m_1,l_1)$. 
Section \ref{sub.thm2} below implies that 
$R_n$ satisfies a linear recursion with constant coefficients in 
$\BQ[m_1,l_1]$. This recursion is valid for all integers $n$ and concludes 
the proof of Theorem \ref{thm.1}.
\qed

\section{Proof of Theorem \ref{thm.3}}
\lbl{sec.thm2}

\subsection{The support function of a polytope}
\lbl{sub.support}

Let us review some standard facts of {\em Polyhedral Geometry}
regarding the {\em support function} $h_P$ of a convex body $P$ in $\BR^r$. 
For a detailed discussion, see \cite[Sec.1.7]{Sc}. The latter is defined by
$$
h_P: \BR^r\setminus\{0\} \longto \BR, \qquad h_P(u)=\sup\{u \cdot x \,|
x \in P\} 
$$
where $u \cdot v$ denotes the standard inner product of two vectors $u$
and $v$ of $\BR^r$. 
Given a unit vector $u$, there is a unique hyperplane with outer normal
vector $u$ that touches $P$, and entirely contains $P$ in the left-half space.
The value $h_P(u)$ of the support function is the signed distance from the 
origin to the above hyperplane. This is illustrated in the following figure:
$$
\psdraw{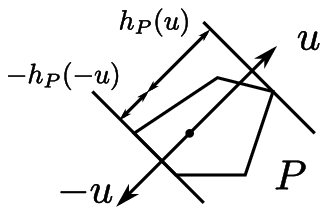}{1.5in}
$$
Let us list some useful properties of the support function:
\begin{itemize}
\item
$h_P$ uniquely determines the convex body $P$. This is the famous {\em
Minkowski reconstruction theorem}. For a detailed proof, see 
\cite[Thm.1.7.1]{Sc} and also \cite{Kl}.
\item
$h_P$ is homogeneous and subadditive.
\item
When $P$ is a convex polytope with vertex set $V_P$, then
\begin{equation}
\lbl{eq.hPmax}
h_P(u)=\max\{u \cdot v\, |v \in V_P\} 
\end{equation}
In particular, $h_P$ is a piece-wise linear function.
\item
The projection of $P$ to the line $\BR u$ is the line segment
$[-h_P(-u),h_P(u)]$. See the above figure for an illustration.
\end{itemize}
Now, consider a polynomial 
$$
R(x)=\sum_{\a \in \calA} c_{\a} x^{\a}
$$
in $r$ variables $(x_1,\dots,x_r)$ where $\calA \subset \BZ^r$ is a finite 
set of exponents, and $c_{\a} \neq 0$. As usual, we abbreviate 
$x^{\a}=\prod_i x_i^{\a_i}$. Let $P$ denote the Newton polytope of $R$.
If $\w \in \BQ^r\setminus\{0\}$, consider the
specialization of variables given by 
\begin{equation}
\lbl{eq.ew}
\e_{\w}(x_1,\dots,x_r)=(t^{\w_1},\dots,t^{\w_r}) 
\end{equation}
Consider the corresponding Laurent polynomial 
$$
R_{\w}(t)=R(\e_{\w}(x))= \sum_{\a \in \calA} c_{\a} t^{ \w \cdot \a}
$$
For generic weight $\w$, the Newton polygon of $R_{\w}(t)$ equals to the
projection of $P$ to the line $\BR \w$. It follows that for generic
$\w$, the Newton polygon of $R_{\w}(t)$ is given by $[-h_P(\w),h_P(\w)]$,
where $h_P$ is piece-wise linear function of $\w$.

\subsection{Generalized power sums and their zeros}
\lbl{sub.power}

Generalized power sums play a key role to the LMS theorem. For a detailed
discussion, see \cite{vP} and also \cite{EvPSW}. Recall that a
{\em generalized power sum} $a_n$ for $n=0,1,2,\dots$ is an expression
of the form
\begin{equation}
\lbl{eq.gps}
a_n=\sum_{i=1}^m A_i(n) \a_i^n
\end{equation}
with {\em roots} $\a_i$, $1 \leq i \leq m$, distinct nonzero quantities,
and coefficients $A_i(n)$ polynomials of degree $m_i-1$ for
positive integers $m_i$, $1 \leq i \leq m$. The generalized power sum $a_n$
is said to have {\em order}
$$
d=\sum_{i=1}^m m_i
$$
and satisfies a linear recursion with constant coefficients of the form
$$
a_{n+d}=s_1 a_{n+d-1} + \dots + s_d a_n
$$
where 
$$
s(x)=\prod_{i=1}^m (1-\a_i x)^{m_i}=1-s_1x-\dots s_d x^d.
$$
It is well-known that a sequence is {\em recurrent} i.e., satisfies a linear 
recursion with constant coefficients if and only if it is a generalized 
power sum. Observe that the monic polynomial polynomial $s(x)$ of smallest
possible degree is uniquely determined by $(a_n)$.

The LMS theorem concerns the zeros of a generalized power sum.

\begin{theorem}
\lbl{thm.MLS}\cite{Sk,Ma,Le}
The zero set of a generalized power sum is the union of a finite set and
a finite set of arithmetic progressions.
\end{theorem}

A detailed proof of this important theorem is discussed in \cite{vP}, for
recurrent sequences with values in an arbitrary field of characteristic
zero. In the next section we will need a slightly stronger form of the LMS
theorem. We say that a recurrent sequence $(a_n)$ is {\em non-degenerate} if
the ratio of two distinct roots of $(a_n)$ is not a root of unity; see 
\cite[Sec.1.1.9]{EvPSW}. 

The LMS theorem in the case of number fields follows from the following 
two theorems.

\begin{theorem}
\lbl{thm.deg1}
\cite[Thm.1.2]{EvPSW}
If $(a_n)$ is reccurrent sequence there exists $M \in \BN$ such that for 
every $r$ with $0 \leq r \leq M-1$, the subsequence $(a_{nM+r})$ is either 
zero or non-degenerate.
\end{theorem}
In fact, if $(a_n)$ takes values in a number field $K$,
there are absolute bounds for $M$ in terms of the degree of $K/\BQ$ and 
the order of $(a_n)$.

\begin{theorem}
\lbl{thm.deg2}
\cite[Cor.1.20]{EvPSW}
If $(a_n)$ is non-degenerate recurrent sequence with values in a number field
$K$, then it has finitely many zero terms.
\end{theorem}
In fact, the number of zeros is bounded above by the degree of $K/\BQ$
and the order of $(a_n)$; see \cite[Eqn.1.18]{ESS}.

\subsection{Proof of Theorem \ref{thm.3}}
\lbl{sub.thm2}

Fix a holonomic sequence $R_n(x_1,\dots,x_r) \in \BQ[x_1^{\pm 1},\dots,
x_r^{\pm 1}]$ (non-zero, for all but finitely many $n$) 
with Newton polytope $N_n$. Thus,
\begin{equation}
\lbl{eq.rR}
\sum_{k=0}^d c_k R_{n+k}=0
\end{equation}
where $c_k \in \BQ[x_1^{\pm 1},\dots, x_r^{\pm 1}]$ for all $k=0,\dots,d$
and $c_d \, c_0 \neq 0$. Suppose first that $r=1$. Let us abbreviate $x_1$ 
by $x$. Consider the characteristic polynomial $p(z,x)$ of \eqref{eq.rR}
$$
p(z,x)= \sum_{k=0}^d c_k(x) z^k= c_d(x)\prod_{j} (z-\l_j(x))^{m_j}
$$
Its roots $\l_j(x)$ (each, with multiplicity $m_j$) are nonzero distinct 
algebraic functions of $x$. The field $E=\overline{\BQ(x)}$
of algebraic functions of $x$ has a {\em valuation} $v^*$ (resp. $v$)
given by the {\em minimum degree} (resp. {\em maximum degree})
with respect to $x$ of a polynomial of $x$, and then 
extended by additivity to the field of rational functions of $x$, and
further uniquely extended to $E$, the field of {\em Puiseux series}
of $x$; see \cite{Wa}. Note that if $R \in E$, then $v^*(R)$ 
is given by by the lowest power of $x$ in the series expansion
of $R(x)$ at $x=0$. Likewise, $v(R)$ is given by the negative of the
lowest power of $x$ in the series expansion of $R(1/x)$ at $x=\infty$. 
For example,
$$
v^*(x^2+x^7)=2, \qquad v(x^2+x^7)=7
$$
Now, the general solution of a linear recursion with constant coefficients 
is of the form
$$
R_n(x)=\sum_j P_j(x,n) \l_j(x)^n 
$$
where $P_j(x,n) \in E[n]$ is a {\em nonzero} polynomial of $n$ of degree 
$m_j-1$. Moreover, 
the Newton polygon $N_n$ of $R_n(x)$ is a line segment $[v^*(R_n),v(R_n)]$. 

Let us concentrate on the valuation $v^*$. 
The series expansion of $\l_j(x)$ and $P_j(x,n)$ at $x=0$ is given by 
\begin{align*}
\l_j(x)&=\a_j x^{\w_j} \left(1+ \sum_{k=1}^\infty c_{j,k} x^{k/r'}\right)
\\
P_j(x,n)&=x^{\b_j} \sum_{k=0}^\infty d_{j,k}(n) x^{k/r'} 
\end{align*}
where $\w_j=v^*(\l_j(x))$ and $\b_j=v^*(P_j(x,n))$. 
We partition the indexing set
$\{j=1,\dots, d\}=J_1 \sqcup J_2 \dots \sqcup J_s$ such that $v^*(\l_j(x))=\w_i$
for all $j \in J_i$ where $\w_1 < \w_2 \dots < \w_s$. 
Without loss of generality, we assume that $r'=1$ and $\b_j=0$
for all $j$, and that the coefficients of the above power series are defined
over a number field $K$. Observe that we can expand 

\begin{equation}
\lbl{eq.nlambda}
\begin{split}
\left(1+ \sum_{k=1}^\infty c_k x^k\right)^n & =1 + n c_1 x + 
\left(n c_2 + \frac{n(n-1)}{2} c_1^2\right) x^2 + \\ &
\left(n c_3 + n(n-1) c_1 c_2 + 
\frac{n(n-1)(n-2)}{6} c_1^3\right) x^3 +
\dots 
\end{split}
\end{equation}
into power series in $x$, where the coefficients are polynomials in $n$.
It follows that
\begin{align*}
R_n(x) &= \sum_{i=1}^s\sum_{k=0}^\infty x^{n \w_i} 
\left(\sum_{j \in J_i} \a_j^n c_{j,k}(n)\right) x^k
\\
&= \sum_{(i,k) \in S}  a_{i,k}(n)  x^{n \w_i+k} 
\end{align*}
where
$$
S=\{1,\dots,s\} \times \BN
$$ 
and
\begin{align*}
c_{j,k}(n) &=\coeff\left( P_j(x,n) \left(\frac{\l_j(x)}{
 \a_j x^{\w_j}}\right)^n, x^k \right) \in K[n] 
\\
a_{i,k}(n)&=\sum_{j \in J_i} \a_j^n c_{j,k}(n)
\end{align*}
Observe that for every $(i,k) \in S$,
$(a_{i,k}(n))$ is a generalized power sum with roots 
in a subset $\{\a_1,\dots,\a_d\}$ of $K^*$. It follows that there exists
a natural number $M$ such that for every $r$ with $0 \leq r \leq M-1$
and every $(i,k) \in S$, the recurrent sequence $a_{i,k}(nM+r)$ is either
zero or non-degenerate. It suffices to show that for every $r$ with
$0 \leq r \leq M-1$, $v^*(R_{nM+r}(x))$ is a linear function of $n$, for
all but finitely many $n$. 

Introduce the well-ordering of $S$ as follows: $(i,k) < (i',k')$ if
$i<i'$ or $i=i'$ and $k < k'$. Since
\begin{equation}
\lbl{eq.RnMk}
R_{nM+r}(x)= \sum_{(i,k) \in S}  a_{i,k}(nM+r)  x^{(nM+r) \w_i+k} 
\end{equation}
is nozero for all but finitely many $n$, it follows that there is a smallest
$(i,k) \in S$ such that $(a_{i,k}(nM+r))$ is not identically zero as a function
of $n$. Since $(a_{i,k}(nM+r))$ is non-degenerate, Theorem \ref{thm.deg2}
implies that $\{n \in \BN\, | a_{i,k}(nM+r)=0\}$ is a finite set, and for all
$n$ in its complement, Equation \eqref{eq.RnMk} implies that
$$
v^*(R_{nM+r}(x))=(nM+r) \w_i+k \,.
$$
It follows that the restriction of $v^*(R_n(x))$ to each arithmetic
progression $M \BN + r$ is a linear function of $n$ 
(for all but finitely many $n$), thus, $v^*(R_n(x))$ (and likewise,
$v(R_n(x))$ is quasi-linear for
all but finitely many values of $n$. 
 
We now reduce the general case of Theorem \ref{thm.3} to the case of
$r=1$. Consider a holonomic sequence 
$R_n(x_1,\dots,x_r) \in \BQ[x_1^{\pm 1},\dots,x_r^{\pm 1}]$ nonzero for all but 
finitely many $n$, and let $N_n$ denote the Newton polytope of $R_n$.
Fix a general weight vector $\w=(\w_1,\dots,\w_r)\in\BQ^r\setminus\{0\}$,
and consider the specialization $R_{\w,n}(t)=R_n(\e_{\w}(x))$. 

To state our next lemma, recall the notion of a {\em fan}, i.e., a 
finite collection of rational polyhedral cones with vertex at the origin,
whose union of their closures covers a fixed vector space, and whose
interiors are pairwise disjoint; see \cite{Zi}. Fans are
also known as constant coefficient tropical varieties; see \cite{SS}.

\begin{lemma}
\lbl{lem.vlambda}
There exists a fan in $\BR^r$ such that for every $j$, the restriction of
$\w \mapsto v^*(\l_j(\e_{\w}(x))$ and $\w \mapsto v(\l_j(\e_{\w}(x))$ 
to the interior of each maximal cone is a linear function.
\end{lemma}

\begin{proof}
Without loss of generality, we will work with $v^*$.
Write the characteristic polynomial $p(z,x)$ in terms of its monomials:
$$
p(z,x)=\sum_{(\a,\b) \in \BZ^r \times \BN} c_{\a,\b} x^{\a} z^{\b}
$$
where the sum is finite and $c_{\a,\b} \neq 0$. Let $\l_j(x)$ satisfy 
$p(\l_j(x),x)=0$. Then, it follows that for each generic $\w$, 
$$
\min_{\a,\b} \{\a \cdot \w + \b v^*(\e_{\w}(\l_j(x)))\}
$$
is achieved at least twice. The result follows. 
\end{proof}
Continuing with the proof of Theorem \ref{thm.3}, 
it is easy
to see that for generic $\w$ (for example, when $c_d(\e_{\w}(x)) \neq 0$
in Equation \eqref{eq.rR}), the sequence $R_{\w,n}(t)$ is holonomic, and
its Newton polygon is given by $[-h_{N_n}(-\w),h_{N_n}(\w)]$, where
$h_{N_n}$ is the support function of $N_n$. This follows from the discussion
of Section \ref{sub.support}. Let $\w$ lie in the interior of a fixed
maximal cone of the fan of Lemma \ref{lem.vlambda}. Lemma \ref{lem.vlambda}
and the proof of the case of $r=1$ implies that for all $n$ sufficiently
large we have
\begin{equation}
\lbl{eq.homega}
-h_{N_n}(-\w)=\d^*(n)\cdot \w, \qquad
h_{N_n}(\w)=\d(n)\cdot \w
\end{equation}
where $\d^*(n)$ and $\d(n)$ are $r$-vectors of quasi-linear functions. 
Of course, $\d^*$ and $\d$ depend on the maximal cone of the above fan.
Now, fix a large enough $n=n_0$. Then, Equations \eqref{eq.hPmax}
and  \eqref{eq.homega} imply that for all $\w$ in the interior
of a maximal cone of the fan we have:
$$
h_{N_n}(\w)=\d(n)\cdot \w=\max\{\w \cdot v\,|v \in V_{N_n}\}
$$
Both are piecewise linear functions of $\w$ and the one on the right
jumps at a hyperplane normal to a facet (i.e., a maximal face) of $N_n$. It
follows that the set of normal vectors of the facets of $N_n$ is a subset
of the set of rays of the fan of Lemma \ref{lem.vlambda}. The latter is
a finite set independent of $n$. 

Consequently, $N_n$ is a sequence of polytopes with normal vectors in
a fixed finite set, and with $h$-function that satisfies Equation 
\eqref{eq.homega}, which is a locally (with respect to $\w$) quasi-linear
with respect to $n$. It follows from the Minknowski recustruction theorem
that the coordinates of every vertex of $N_n$ are quasi-linear functions
of $n$, for all but finitely many $n$. This concludes the proof
of Theorem \ref{thm.3}.
\qed

\begin{remark}
\lbl{rem.3}
The LMS theorem is used in the same way both in the proof of Theorem 
\ref{thm.3}, and in the proof of \cite[Thm.1]{Ga4}. This is not a coincidence.
In fact, a holonomic sequence $R_n(q) \in \BQ(q)$ of one variable is
also a $q$-holonomic sequence in the sense of \cite{Z,Ga4}, since it satisfies
a linear recursion with constant coefficients. The degree of a $q$-holonomic
sequence is a quadratic quasi-polynomial (see \cite{Ga4}), however
in the case of constant coefficients the slopes of the corresponding 
tropical curve are zero, and the degree is a linear quasi-polynomial.
Thus, the $r=1$ case of Theorem \ref{thm.3} follows from \cite[Thm.1]{Ga4}.
For completeness, we gave a proof of Theorem \ref{thm.3} when $r=1$
that avoids the general machinery of \cite{Ga4}.
\end{remark}


\bibliographystyle{hamsalpha}\bibliography{biblio}
\end{document}